\documentclass[12pt] {amsart}
\usepackage{amsmath,amssymb,amsbsy,amsfonts,latexsym,amsopn,amstext,
                                               amsxtra,euscript,amscd,bm}
\usepackage{color}                                               

\usepackage[pdftex]{hyperref}
\hypersetup{pdfstartview=FitH,	pdfpagemode=UseNone}

\newfont{\teneufm}{eufm10}
\newfont{\seveneufm}{eufm7}
\newfont{\fiveeufm}{eufm5}
\newfam\eufmfam
                \textfont\eufmfam=\teneufm \scriptfont\eufmfam=\seveneufm
                \scriptscriptfont\eufmfam=\fiveeufm
\def\bbbc{{\mathchoice {\setbox0=\hbox{$\displaystyle\rm C$}\hbox{\hbox
to0pt{\kern0.4\wd0\vrule height0.9\ht0\hss}\box0}}
{\setbox0=\hbox{$\textstyle\rm C$}\hbox{\hbox
to0pt{\kern0.4\wd0\vrule height0.9\ht0\hss}\box0}}
{\setbox0=\hbox{$\scriptstyle\rm C$}\hbox{\hbox
to0pt{\kern0.4\wd0\vrule height0.9\ht0\hss}\box0}}
{\setbox0=\hbox{$\scriptscriptstyle\rm C$}\hbox{\hbox
to0pt{\kern0.4\wd0\vrule height0.9\ht0\hss}\box0}}}}
\def\bbbq{{\mathchoice {\setbox0=\hbox{$\displaystyle\rm
Q$}\hbox{\raise 0.15\ht0\hbox to0pt{\kern0.4\wd0\vrule
height0.8\ht0\hss}\box0}} {\setbox0=\hbox{$\textstyle\rm
Q$}\hbox{\raise 0.15\ht0\hbox to0pt{\kern0.4\wd0\vrule
height0.8\ht0\hss}\box0}} {\setbox0=\hbox{$\scriptstyle\rm
Q$}\hbox{\raise 0.15\ht0\hbox to0pt{\kern0.4\wd0\vrule
height0.7\ht0\hss}\box0}} {\setbox0=\hbox{$\scriptscriptstyle\rm
Q$}\hbox{\raise 0.15\ht0\hbox to0pt{\kern0.4\wd0\vrule
height0.7\ht0\hss}\box0}}}}
\def\bbbt{{\mathchoice {\setbox0=\hbox{$\displaystyle\rm
T$}\hbox{\hbox to0pt{\kern0.3\wd0\vrule height0.9\ht0\hss}\box0}}
{\setbox0=\hbox{$\textstyle\rm T$}\hbox{\hbox
to0pt{\kern0.3\wd0\vrule height0.9\ht0\hss}\box0}}
{\setbox0=\hbox{$\scriptstyle\rm T$}\hbox{\hbox
to0pt{\kern0.3\wd0\vrule height0.9\ht0\hss}\box0}}
{\setbox0=\hbox{$\scriptscriptstyle\rm T$}\hbox{\hbox
to0pt{\kern0.3\wd0\vrule height0.9\ht0\hss}\box0}}}}
\def\bbbs{{\mathchoice
{\setbox0=\hbox{$\displaystyle     \rm S$}\hbox{\raise0.5\ht0\hbox
to0pt{\kern0.35\wd0\vrule height0.45\ht0\hss}\hbox
to0pt{\kern0.55\wd0\vrule height0.5\ht0\hss}\box0}}
{\setbox0=\hbox{$\textstyle        \rm S$}\hbox{\raise0.5\ht0\hbox
to0pt{\kern0.35\wd0\vrule height0.45\ht0\hss}\hbox
to0pt{\kern0.55\wd0\vrule height0.5\ht0\hss}\box0}}
{\setbox0=\hbox{$\scriptstyle      \rm S$}\hbox{\raise0.5\ht0\hbox
to0pt{\kern0.35\wd0\vrule height0.45\ht0\hss}\raise0.05\ht0\hbox
to0pt{\kern0.5\wd0\vrule height0.45\ht0\hss}\box0}}
{\setbox0=\hbox{$\scriptscriptstyle\rm S$}\hbox{\raise0.5\ht0\hbox
to0pt{\kern0.4\wd0\vrule height0.45\ht0\hss}\raise0.05\ht0\hbox
to0pt{\kern0.55\wd0\vrule height0.45\ht0\hss}\box0}}}}
\def\bbbz{{\mathchoice {\hbox{$\sf\textstyle Z\kern-0.4em Z$}}
{\hbox{$\sf\textstyle Z\kern-0.4em Z$}} {\hbox{$\sf\scriptstyle
Z\kern-0.3em Z$}} {\hbox{$\sf\scriptscriptstyle Z\kern-0.2em
Z$}}}}

\def \balpha{\bm{\alpha}}
\def \bbeta{\bm{\beta}}
\def \bgamma{\bm{\gamma}}

\newtheorem{theorem}{Theorem}
\newtheorem{lemma}[theorem]{Lemma}

\newtheorem{cor}[theorem]{Corollary}

\def\squareforqed{\hbox{\rlap{$\sqcap$}$\sqcup$}}
\def\qed{\ifmmode\squareforqed\else{\unskip\nobreak\hfil
\penalty50\hskip1em\null\nobreak\hfil\squareforqed
\parfillskip=0pt\finalhyphendemerits=0\endgraf}\fi}

\def\cA{{\mathcal A}}

\def\cC{{\mathcal C}}

\def\cG{{\mathcal G}}

\def\cI{{\mathcal I}}

\def\cO{{\mathcal O}}
\def\cP{{\mathcal P}}

\def\cS{{\mathcal S}}

\def\cU{{\mathcal U}}
\def\cV{{\mathcal V}}

\def \sf {\mathfrak s}

\def\Oes{\cO_{e,s}}

\def\Oet{\cO_{e,t}}

\def\fO{\mathfrak{O}}

\def\Oef{\fO_{e,f}}
\def\Oeg{\fO_{e,g}}

\def\Im{{\mathrm{Im}}}

\newcommand{\ignore}[1]{}

\def\vec#1{\mathbf{#1}}

\def \C{\mathbb{C}}
\def \F{\mathbb{F}}
\def \K{\mathbb{K}}

\def \Z{\mathbb{Z}}

\def \Q{\mathbb{Q}}

\def \Z{\mathbb{Z}}

\def\mand{\qquad\mbox{and}\qquad}

\def\\{\cr}
\def\({\left(}
\def\){\right)}
\def\fl#1{\left\lfloor#1\right\rfloor}
\def\rf#1{\left\lceil#1\right\rceil}

\title[Interpolation and Identity Testing 
from Powers]{Polynomial Interpolation and Identity Testing 
from High Powers over Finite Fields}

\author[Ivanyos et al.]{G{\'a}bor Ivanyos}
\address{Institute for Computer Science and Control, Hungarian Academy of Sciences,     H-1111 Budapest, Hungary}
\email{gabor.ivanyos@sztaki.mta.hu}

\author[{}]{Marek Karpinski}
\address{Department of Computer Science, Bonn University, 
 53113 Bonn, Germany} 
\email{marek@cs.uni-bonn.de}

\author[{}]{Miklos Santha}
\address{CNRS, Universit{\'e} Paris Diderot, 75013 Paris, France and CQT, National University of Singapore, 117543 Singapore}
\email{miklos.santha@liafa.univ-paris-diderot.fr}

\author[{}]{Nitin Saxena}
\address{Department of Computer Science and Engineering, 
 Indian Institute of Technology, 
Kanpur, UP 208016, India}
\email{nitin@cse.iitk.ac.in}

\author[{}]{Igor E.~Shparlinski}
\address{Department of Pure Mathematics, University of 
New South Wales, Sydney, NSW 2052 Australia}
\email{igor.shparlinski@unsw.edu.au}

\begin{document}

\keywords{hidden polynomial power, black-box interpolation, Nullstellensatz,  rational function, determistic algorithm, randomised algorithm, quantum algorithm}

\subjclass{11T06,  11Y16, 68Q12, 68Q25}

\begin{abstract} 
We consider the problem of recovering (that is, interpolating) and identity testing of a ``hidden'' monic polynomial $f$, given an oracle access to $f(x)^e$ for $x\in\F_q$ (extension fields access is not permitted). The naive 
interpolation algorithm needs $O(e \deg f)$ queries and thus requires $e\deg f<q$. We design algorithms that are asymptotically better in certain cases; requiring only $e^{o(1)}$ queries to the oracle.
In the randomized (and quantum)
setting, we give a substantially better interpolation algorithm, that requires only $O(\deg f \log q)$ queries. Such results have been known before only for the special case of a linear $f$, called the {\em hidden shifted power} problem.

We use techniques from algebra, such as effective versions of Hilbert's Nullstellensatz, and analytic number theory, such as results on the distribution  of rational functions in subgroups and character sum estimates.
\end{abstract}

 \maketitle

\section{Introduction} 

Let $\F_q$ be  a finite field of  $q$ elements. 
Here we consider several problems of recovering 
and identity testing of 
a ``hidden'' monic polynomial $f\in \F_q[X]$, given  
$\Oef$  an oracle that on every input $x \in \F_q$ outputs
$\Oef(x) = f(x)^e$ for some large positive integer $e\mid q-1$. 

More precisely, we consider the following problem 
{\it Interpolation  from Powers\/}:
\begin{quote}
given an oracle  $\Oef$ for some unknown  monic polynomial $f\in \F_q[X]$,
 recover $f$.
\end{quote}

We also consider the following two versions of  the
{\it Identity Testing from Powers\/}:
\begin{quote}
given an oracle  $\Oef$ for some unknown monic polynomial $f\in \F_q[X]$
and another known polynomial $g\in \F_q[X]$,  decide whether $f = g$,
\end{quote}
and
\begin{quote}
given two oracles  $\Oef$ and $\Oeg$ for some unknown monic 
polynomials $f,g\in \F_q[X]$,  decide whether $f = g$.\end{quote}

In particular, for a linear polynomial 
$f(X) = X+s$, with a `hidden' $a\in \F_q$, 
we denote  $\Oef = \Oes$.  We remark that in this case 
there are two naive algorithms that work for linear polynomials:
\begin{itemize}

\item One can query $\Oes$ at $e+1$ arbitrary points and then 
using a  fast interpolation
algorithm, see~\cite{vzGG}, get a deterministic algorithm of
complexity $e (\log q)^{O(1)}$
(as in~\cite{vzGG}, we measure the complexity of
an algorithm by the number of bit operations in the standard RAM model).

\item For probablistic testing one can  query $\Oes$ (and $\Oet$) at randomly chosen elements $x \in \F_q$ until the desired level of confidence is achieved (note that the equation $(x+s)^e = (x+t)^e$ has at most $e$ solutions $x \in \F_q$). 

\end{itemize}

These naive algorithms have been improved by Bourgain, Garaev, Konyagin and Shparlinski~\cite{BGKS} in several cases (with respect to both the time complexity and the number of queries).

Furthermore, in the case when a quantum version of the 
oracle $\Oes$ is given, van Dam,  Hallgren and 
Ip~\cite{vDamHallgIp} have given a polynomial time 
quantum algorithm which recovers $s$, see also~\cite{vDam}.

For non-linear polynomials $f\in \F_q[X]$ some classical 
and quantum algorithms are given by Russell and 
Shparlinski~\cite{RusShp}. However they do not reach the 
level of those of~\cite{BGKS,vDam,vDamHallgIp} due to 
several additional obstacles which arise for non-linear 
polynomials. For example, we note that both the interpolation 
and random sampling algorithms fail if $e \deg f > q$. 
Indeed, note that queries from the extension field are not permitted,
and $\F_q$ may not have enough elements to make these algorithms 
correct. 

Here we consider both classical and quantum algorithms.
In particular, we extend the results of~\cite[Section~3.3]{BGKS}
to arbitrary monic polynomials $f \in \F_p[X]$ for a prime $p$. 
These results are based on 
some bounds of character sums and also new results about 
the order of multiplicative group generated by the values of a
rational function on several consecutive integers.

Further, we also consider quantum algorithms. However, our setting 
is quite different from those of~\cite{vDam,vDamHallgIp}  as
we do not assume that the values of $f$ are given by a quantum 
oracle, rather the algorithm works with the classical oracle $\Oef$.

The above questions appear naturally in understanding the pseudorandomness of the {\em Legendre symbol} $\left(\frac{f(x)}{p}\right)$. In particular, this has applications in the cryptanalysis of certain homomorphic cryptosystems. See \cite{BM, BL, Dam,   MvOV} for further details.

Note that the above questions are closely related to the 
general problem of oracle (also sometimes called ``black-box'') polynomial interpolation
and identity testing for arbitrary  polynomials (though forbidding the use of field extensions makes the problems harder),
see~\cite{Sax09,Sax14,SY} and the references therein.

Throughout the paper, any implied constants in the symbols $O$,
$\ll$  and $\gg$  may occasionally, where obvious, 
depend on the degree $d$ of the polynomial $f$ (\& an integer parameter $\nu$), and are absolute otherwise. 
We recall that the
notations $U = O(V)$, $U \ll V$ and  $V \gg U$ are all equivalent to the
statement that the inequality $|U| \le c V$ holds with some
constant $c> 0$.

\section{Identity Testing on Classical Computers}
\label{sec:class ident test}
\subsection{Main results}

Here we consider the identity testing case of two unknown {\em monic} 
polynomials $f,g\in \F_q[X]$ of degree $d$ given 
the oracles $\Oef$ and $\Oeg$. We remark that if $f/g$ is
an $(q-1)/e$-th power of a 
nonconstant 
 rational function 
over $\F_q$ then  it is impossible to distinguish 
between $f$ and $g$ from the oracles $\Oef$ and $\Oeg$.
We write $f\sim_e g$ in this case, and $f\not \sim_e g$
otherwise. 

We note that it is shown in the proof of~\cite[Theorem~6]{RusShp}
that the Weil bound of multiplicative character sums 
(see~\cite[Theorem~11.23]{IwKow})
implies that given two oracles  $\Oef$ and $\Oeg$ for some unknown monic 
polynomials $f,g\in \F_q[X]$  with $f\not \sim_e g$ one can  decide 
whether $f = g$
in time $q^{1/2+o(1)}$. Note that the result of~\cite{RusShp}
is stated only for prime fields $\F_p$ but it can be extended to
arbitrary fields at the cost of only typographical changes. The same holds for 
here the results of Section~\ref{sec:rand_quant interp} but the results 
 of Section~\ref{sec:class ident test} hold only for prime fields.

For ``small'' values of $e$, over prime fields $\F_p$, we have a stronger result.

\begin{theorem}[Small $e$]
\label{thm:Small_e} For a prime $p$ and a positive integer
$e\mid p-1$, with $e\le p^{\delta}$ for some fixed
$\delta >0$, given two oracles $\Oef$ and $\Oeg$  for some 
unknown monic polynomials $f,g\in \F_p[X]$ of degree $d$ 
with $f\not \sim_e g$, 
there is a deterministic algorithm to decide 
whether $f = g$ in $e^{c_0(d)\delta^{1/(2d-1)}}$ queries 
to  the oracles $\Oef$ and $\Oeg$, where $c_0(d)$ depends only on $d$.
\end{theorem}

In particular, we see from (the proof of)  Theorem~\ref{thm:Small_e} 
that if $e= p^{o(1)}$ and $e\to\infty$ then we can test whether $f=g$ in
time $e^{o(1)} (\log p)^{O(1)}$ in $e^{o(1)}$ oracle calls.

For intermediate values of $e$, the following result  complements 
 both Theorem~\ref{thm:Small_e} and the result of~\cite{RusShp}.
We, however, have to assume that  the polynomials $f$ and $g$ 
are {\em irreducible}.

\begin{theorem}[Medium $e$]
\label{thm:Medium_e} For a prime $p$ and a positive integer
$e\mid p-1$, with $e\le p^{\eta-\delta}$ for some fixed
$\delta >0$, given two oracles $\Oef$ and $\Oeg$  for some 
unknown  monic  
 polynomials $f,g\in \F_p[X]$ of 
degree $d\ge 1$ with $f\not \sim_e g$, 
there is a deterministic algorithm to decide 
whether $f = g$ in $e^{\kappa+\delta}$ queries to 
the oracles $\Oef$ and $\Oeg$, where
$$ 
\eta = \frac{4d-1}{4d^2(d+1)^2} \mand
\kappa =  \frac{2d}{4d-1}. 
$$
\end{theorem}

The proofs of Theorems~\ref{thm:Small_e} and~\ref{thm:Medium_e} are 
given below in Sections~\ref{sec: proof small e} and~\ref{sec: proof med e},
respectively.

\subsection{Background from arithmetic algebraic geometry}

Our argument makes use of a slight modification of~\cite[Lemma~23]{BGKS}, 
which is based on a quantitative 
version of effective Hilbert's Nullstellensatz given 
by  D'Andrea,  Krick and  Sombra~\cite{DKS}, which
improved the previous estimates due  to Krick,  Pardo and   Sombra~\cite{KiPaSo}.

As usual, we define the {\em logarithmic height} of a nonzero polynomial $P \in
\Z[Z_1, \ldots, Z_n]$   as the maximum
logarithm of the largest (by absolute value) coefficient of $P$.

The next statement is essentially~\cite[Lemma~23]{BGKS}, 
however we now use~\cite[Theorem~2]{DKS} instead of~\cite[Theorem~1]{KiPaSo}.

\begin{lemma}
\label{lem:Hilb} 
Let $P_1, \ldots, P_N \in \Z[Z_1, \ldots,
Z_n]$ be $N\ge 2$ polynomials in $n$ variables of degree at most
$D\ge 3$ and of logarithmic height at most $H$ 
and let  $R \in \Z[Z_1, \ldots,
Z_n]$ be a polynomial in $n$ variables of degree at most
$d\ge 3$ and of logarithmic height at most $h$  such that $R$ 
vanishes on the variety
$$
P_1(Z_1, \ldots, Z_n) = \ldots =P_N(Z_1, \ldots, Z_n) = 0.
$$
There are polynomials $Q_1, \ldots, Q_N\in \Z[Z_1, \ldots, Z_n]$ and  positive integers $A$ and $r$
with
$$
\log A \le  2(n+1)dD^{n}H +   3D^{n+1}h + C(d,D,n,N),$$
such that
$$
P_1Q_1+ \ldots + P_NQ_N = AR^r,
$$
where $C(d,D,n,N)$  depends only on $d$, $D$, $n$ 
and $N$.
\end{lemma}

We note that using Lemma~\ref{lem:Hilb} in the argument of~\cite{BGKS}
allows to replace $\nu^{-4}$ with $\nu^{-3}$ in~\cite[Lemma~35]{BGKS}.
In turn, this allows us to replace $\delta^{1/3}$ with  $\delta^{1/2}$ 
in~\cite[Lemma~38 and Theorem~51]{BGKS}.

We now define the logarithmic height of an algebraic number
$\alpha \ne 0$  as the logarithmic height 
of its minimal polynomial.

We need a slightly more general form of a result of
Chang~\cite{Chang}. In fact, this is exactly the statement that
is established in the proof of~\cite[Lemma~2.14]{Chang},
see~\cite[Equation~(2.15)]{Chang}.

\begin{lemma}
\label{lem:SmallZero} Let $P_1, \ldots, P_N, R \in \Z[Z_1, \ldots,
Z_n]$ be $N+1 \ge 2$ polynomials in $n$ variables of degree at
most $D$ and of logarithmic height at most $H\ge1$. If  the
zero-set
$$
P_1(Z_1, \ldots, Z_n) = \ldots =P_N(Z_1, \ldots, Z_n) = 0 \quad
\text{and}\quad
R(Z_1, \ldots, Z_n) \ne  0
$$
is not empty then it has a point $(\beta_1, \ldots, \beta_n)$
in an extension $\K$ of $\Q$ of degree $[\K:\Q]\le C_1(D,n)$
such that its logarithmic height is at most $C_2(D,n,N) H$,
where $C_1(D,n)$ depends only on $D$, $n$ and $C_2(D,n,N)$ depends
only on $D$, $n$ and $N$.
\end{lemma}

 \subsection{Product sets in number fields}
For a set $\cA$ in an arbitrary semi-group, we use  $\cA^{(\nu)}$ to denote the {\em $\nu$-fold product set}, that is
$$
\cA^{(\nu)}=\{a_1\ldots a_\nu~:~ a_1,\ldots, a_\nu \in \cA\}.
$$

We  recall the following  result given in~\cite[Lemma~29]{BGKS}, 
which in turn generalises~\cite[Corollary~3]{BKS}.

\begin{cor}
\label{cor:ProdSet} Let $\K$ be a finite extension of $\Q$ of
degree $D = [\K:\Q]$. Let $\cC \subseteq \K$ be a finite set with
elements of  logarithmic height at most $H\ge2$.
Then we have
$$
\#\cC^{(\nu)} > \exp\(-c(D,\nu) \frac{H}{\sqrt{\log H}}\) (\#\cC)^\nu,
$$
where $c(D,\nu)$ depends only on $D$ and $\nu$.
\end{cor}

 \subsection{Product sets of consecutive values of rational functions in prime fields}
 
We now show that for a nontrivial rational function $f/g \in \F_p(X)$
and an integer $h\ge 1$, the set 
formed by $h$ consecutive values of $f/g$ cannot be all inside 
a small multiplicative subgroup $\cG \subseteq \F_p^*$. 
For the linear fractional function $(X+s)/(X+t)$ this has been obtained in~\cite[Lemma~35]{BGKS}.

\begin{lemma}
\label{lem:ProdSet} 
Let $\nu \ge 1$ be a fixed integer.
Assume
that for some sufficiently large positive integer $h$ and prime
$p$ we have
$$
h < p^{c(d)\nu^{-2d}},
$$
where $c(d)$ depends only on $d$. 
For two distinct monic 
polynomials $f,g\in \F_p$ of degrees $d$,   we consider the set
$$
\cA = \left\{\frac{f(x)}{g(x)}~:~ 1\le x\le  h \right\} \subseteq
\F_p.
$$
Then
$$
\# \cA^{(\nu)} >  \exp\(-c(d,\nu) \frac{\log h  }{\sqrt{\log \log h }}\) h ^{\nu},
$$
where  $c(d,\nu)$ depends only on $\nu$ and $d$. 
\end{lemma}

\begin{proof}  We closely follow the proof of~\cite[Lemma~35]{BGKS}. 
Let 
$$f(X)= X^d + \sum_{k=0}^{d-1} a_{d-k} x^k
\mand
g(X) = X^d + \sum_{\ell=0}^{d-1} b_{d-\ell} X^\ell.
$$ 
The idea is to move from the finite field to a number field, where 
we are in a position to apply Corollary~\ref{cor:ProdSet}.

We consider the collection $\cP\subseteq 
\Z[\vec{U}, \vec{V}]$, where 
$$
\vec{U} = (U_1, \ldots, U_d)
\mand  \vec{V} =(V_1,\ldots, V_d),
$$
of   polynomials
\begin{equation*}
\begin{split}
P_{\vec{x},\vec{y}}(\vec{U}, \vec{V}) =
\prod_{i=1}^\nu & \(x_i^d + \sum_{k=0}^{d-1} U_{d-k} x_i^k\)
 \(y_i^d + \sum_{\ell=0}^{d-1} V_{d-\ell} y_i^\ell\)\\
& -  \prod_{i=1}^\nu 
\(x_i^d + \sum_{\ell=0}^{d-1} V_{d-\ell} x_i^\ell\) \(y_i^d + \sum_{k=0}^{d-1} U_{d-k} y_i^k\),
\end{split}
\end{equation*}
where $\vec{x} = (x_1, \ldots, x_\nu)$ and
$\vec{y} = (y_1, \ldots, y_\nu)$ are integral vectors with
entries in $\cI:=[1,h]$ and such that
$$
P_{\vec{x},\vec{y}}(x_1, \ldots, x_d,y_1,\ldots, y_d) \equiv 0 \pmod p.
$$
Note that 
$$
P_{\vec{x},\vec{y}}(a_1, \ldots, a_d,b_1,\ldots, b_d) 
\equiv \prod_{i=1}^\nu  f(x_i)g(y_i)
-\prod_{i=1}^\nu  f(y_i)g(x_i) \pmod p.
$$

Clearly if $P_{\vec{x},\vec{y}}$ is identical to zero then, by the 
uniqueness of polynomial factorisation in the ring $ \F_p[\vec{U}, \vec{V}]$, 
the components of $\vec{y}$ are permutations of those of $\vec{x}$. 
So in this case we obviously obtain 
$$
\# \cA^{(\nu)} \ge \frac{1}{\nu!} \(\# f(\cI)\)^{\nu} \gg H^\nu.
$$
Hence, we now assume that $\cP$ contains non-zero polynomials.

Clearly,  every $P\in \cP$ is of degree at most $2\nu$
and of logarithmic height $O(\log h)$.

We take a family $\cP_0$ containing the largest possible number
$$
N \le (\nu +1)^{2d} - 1
$$
of linearly independent polynomials $P_1, \ldots, P_N \in \cP$,  and consider the
variety
$$
\cV: \ \{(\vec{U}, \vec{V}) \in \C^{2d}~:~P_1(\vec{U}, \vec{V}) = \ldots =P_N(\vec{U}, \vec{V}) = 0\}.
$$
Clearly $\cV \ne \emptyset$ as it contains the diagonal $\vec{U} = \vec{V}$. 

We claim that $\cV$ contains a point outside of the diagonal, 
that is, there is a point $(\bbeta,\bgamma)$ 
with $\bbeta , \bgamma \in \C^{d}$ and $\bbeta \ne \bgamma$.

Assume that $\cV$ does not contain a point outside of the diagonal.
Then for every $k=1, \ldots, d$, the polynomial 
$$
R_k(U_1, \ldots, U_d, V_1,\ldots, V_d) = U_k-V_k
$$
vanishes on $\cV$. 

Then by
Lemma~\ref{lem:Hilb} we see that there are polynomials
$Q_{k,1}, \ldots, Q_{k,N} \in \Z[\vec{U}, \vec{V}]$ and  positive integers
$A_k$ and $r_k$ with
\begin{equation}
\label{eq:b small birat}
\log A_k \le  c_0 d (2\nu)^{2d} \log h
\end{equation}
for some absolute constant $c_0$ (provided that $h$ is large enough)
and such that
\begin{equation}
\label{eq:HN UV}
P_1Q_{k,1}+ \ldots + P_NQ_{k,N} = A_k(U_k-V_k)^{r_k}.
\end{equation}
Since $f\ne g$, there is $k\in \{1, \ldots, d\}$ for which 
$a_k \not \equiv b_k \pmod p$. For this $k$ we substitute 
$$
(\vec{U}, \vec{V}) = (a_1, \ldots, a_d,b_1,\ldots, b_d)
$$
in~\eqref{eq:HN UV}. 
Recalling the definition of the set $\cP$ we now derive 
that $p \mid A_k$. Taking 
$$
c(d) = \frac{1}{c_0d2^d+1}
$$ 
in the condition of the lemma, we see
from~\eqref{eq:b small birat} that this is impossible.

Hence  the set
$$
\cU = \cV \cap [\vec{U} - \vec{V} \ne 0]
$$
is nonempty. Applying Lemma~\ref{lem:SmallZero}
we see that  it has a point $(\bbeta, \bgamma)$
with components of   logarithmic height $O(\log h)$
in an extension $\K$ of $\Q$ of degree $[\K:\Q] \le \Delta(d,\nu)$,
where $\Delta(d,\nu)$ depends only on $d$ and $\nu$.

Consider the maps
$\Phi:\  \cI^\nu  \to \F_p$
given by
$$
\Phi: \ \vec{x} = (x_1, \ldots, x_\nu) \mapsto \prod_{j=1}^\nu \frac{f(x_j)}{g(x_j)}
$$
and $\Psi:  \cI^\nu  \to \K$
given by
$$
\Psi: \ \vec{x} = (x_1, \ldots, x_\nu) \mapsto \prod_{j=1}^\nu 
\frac{F_{\bbeta}(x_j)}{G_{\bgamma}(x_j)},
$$
where 
$$F_{\bbeta}(X)= X^d + \sum_{k=0}^{d-1} \beta_{d-k} x^k
\mand
G_{\bgamma}(X) = X^d + \sum_{\ell=0}^{d-1} \gamma_{d-\ell} X^\ell.
$$
By construction of  $(\bbeta, \bgamma)$ we have that
$\Psi(\vec{x}) = \Psi(\vec{y})$ if $\Phi(\vec{x}) =
\Phi(\vec{y})$.
Hence
$$
\# \cA^{(\nu)} \ge \Im \Psi = \# \cC^{(\nu)},
$$
where $\Im \Psi$ is the image set of the map $\Psi$ and
$$
\cC = \left\{\frac{F_{\bbeta}(x)}{G_{\bgamma}(x)}~:~ 1\le x\le h \right\}
\subseteq \K.
$$
Using Corollary~\ref{cor:ProdSet}, we derive the result.
\end{proof}

We also recall the following bound which is a special case of a more general result from~\cite[Theorem~7]{GomShp}.

\begin{lemma}
\label{lem:ValSet}   If for two relatively prime  monic 
polynomials $f,g\in \F_p$ of degree $d\ge 1$, a  positive integer $h$
and a multiplicative subgroup $\cG \subseteq \F_p^*$ we have 
$$
\left\{\frac{f(x)}{g(x)}~:~ 1\le x\le  h \right\} \subseteq \cG.
$$
Then
$$
\#\cG \gg \min\{h^{2(1 -  \tau)+o(1)}, 
h^{2(1-\rho - \tau) +o(1)}p^{2\vartheta} \},
$$ 
where 
$$
\vartheta = \frac{1}{2d(d+2)}, \qquad  \rho = \frac{(d+1)^2}{2(d+2)}, 
\qquad \tau = \frac{1}{4d},
$$
and  the implied constant depends on $d$.
\end{lemma}

\begin{proof}
By~\cite[Theorem~7]{GomShp}, applied with $d = e$ (and thus with 
$k = d(d+1)^2$, $s = d^2+2d$ and hence the above values of $\vartheta$, 
$\rho$ and $\tau$), we have 
$$
  \# \(\left\{\frac{f(x)}{g(x)}~:~ 1\le x\le  h \right\} \bigcap \cG\)
\le \(1 + h^{\rho}p^{-\vartheta}\) h^{\tau + o(1)} T^{1/2}
$$
where $T = \# \cG$. Under the condition of the lemma we have 
$$
 \# \(\left\{\frac{f(x)}{g(x)}~:~ 1\le x\le  h \right\} \bigcap \cG\) = h
$$
and the result follows. 
\end{proof}

\subsection{Proof of Theorem~\ref{thm:Small_e}}
\label{sec: proof small e}
We set 
$$
\nu = \fl{\(\frac{c(d)}{2 \delta}\)^{2d-1}} \mand h = \fl{e^{1/\nu}}+1, 
$$
where $c(d)$ is the constant of Lemma~\ref{lem:ProdSet}. 
We note that
$$
\frac{2\delta }{\nu} \le \frac{c(d)}{\nu^{2d}} 
$$
so as $e \to \infty$ we have
\begin{equation}
\label{eq:h small}
e^{1/\nu} < h  =  e^{1/\nu+o(1)} \le e^{2/\nu} \le p^{2\delta/\nu} \le  p^{c(d)/\nu^{2d}}. 
\end{equation}

We now query the oracles $\Oef$ and $\Oeg$  for $x =1, \ldots, h$. 

If the oracles return two distinct values then clearly $f\ne g$.
Now assume 
$$
f(x)^e = g(x)^e, \qquad x =1, \ldots, h.
$$
Therefore, the values 
$f(x)/g(x)$, $x =1, \ldots, h$ belong to the 
subgroup $\cG_e$ of $\F_p^*$ of order $e$. 
Hence for the set
\begin{equation}
\label{eq:set A}
\cA = \left\{\frac{f(x)}{g(x)}~:~ 1\le x\le  h \right\} \subseteq
\F_p
\end{equation}
for any integer  $\nu \ge 1$  we have 
\begin{equation}
\label{eq:A Ge}
\cA^{(\nu)}=\{a_1\ldots a_\nu~:~ a_1,\ldots, a_\nu \in \cA\}  \subseteq \cG_e.
\end{equation}

We see from~\eqref{eq:h small} that Lemma~\ref{lem:ProdSet} applies
which contradicts~\eqref{eq:A Ge} as we have $h^\nu > e$ 
for the above choice of the parameters. This concludes the  proof.

\subsection{Proof of Theorem~\ref{thm:Medium_e}}
\label{sec: proof med e}

We fix some $\varepsilon > 0$ and set
$$
h = \rf{e^{(1 +\varepsilon)/(2 -  2\tau))}}.
$$
We also 
note that for the above choice of $h$ 
and  for 
\begin{equation}
\label{eq:e small1}
e^{1+\varepsilon} \le e^{(1-\rho - \tau)(1 +\varepsilon)/(1 -  \tau)}p^{\vartheta}
\end{equation}
we have
$$
\min\{h^{2(1 -  \tau)}, 
h^{2(1-\rho - \tau)}p^{2\vartheta} \}\ge e^{1+\varepsilon} . 
$$
Therefore, under the condition~\eqref{eq:e small1},  we derive  from Lemma~\ref{lem:ValSet} 
that for the set $\cA$ given by~\eqref{eq:set A} we have
$\cA \not \subseteq \cG_e$. Proceeding as in the proof of 
Theorem~\ref{thm:Small_e}, we obtain an algorithm that requires $h$ queries. 

Clearly, for the above choice of $h$, the condition~\eqref{eq:e small1}
is satisfied if 
\begin{equation}
\label{eq:e small2}
e^{(1+\varepsilon)\rho/(1 -  \tau)} \le p^{\vartheta}.
\end{equation}
Taking 
$$
\eta =  \frac{\vartheta(1 -  \tau)}{\rho}\mand \kappa = \frac{1}{2 -  2\tau}
$$
we see that the condition~\eqref{eq:e small2} is equivalent to 
$e \le p^{\eta/(1 + \varepsilon)}$, under which 
we get an algorithm which requires $h = O\(e^{(1 +\varepsilon)\kappa}\)$ 
queries.
Since   $\varepsilon>0$ is arbitrary, the result now follows.

\section{Quantum and Randomized Interpolation}
\label{sec:rand_quant interp}
\subsection{Main results}

Here we present a quantum algorithm for the interpolation problem
of finding an unknown monic 
polynomial $f\in \F_q[X]$ of degree $d$ given 
the oracle  $\Oef$. We emphasise the difference between our 
settings where the oracle is classical and only the algorithm is
quantum and the settings of~\cite{vDam,vDamHallgIp} which employ 
the quantum analogue of the  oracle  $\Oef$.

We recall that the oracle $\Oef$ does not accept queries from 
field extensions of $\F_q$, and therefore, if $de > q$, we {\em cannot} interpolate $f^e$ from queries to  $\Oef$.

\begin{theorem}
\label{thm:Quant} 
Given an oracle $\Oef$ for some 
unknown 
monic
polynomial
$f$  of degree at most $d$,  for any $\varepsilon > 0$
there is a quantum algorithm to 
find with probability $1-\varepsilon$
 a polynomial $g$ such that $g\sim_e f$ 
in time $e^{d/2}\(d \log q  \log (1/\varepsilon)\)^{O(1)}$ and 
$O\(d\log q \log (1/\varepsilon)\)$ calls to $\Oef$.
\end{theorem}

Replacing quantum parts of the algorithm above
with classical (randomized) methods,
we obtain the following.

\begin{theorem}
\label{thm:Rand} 
Given an oracle $\Oef$ for some unknown monic polynomial
$f$  of degree at most $d$,  for any $\varepsilon > 0$
there is a {\em randomized} algorithm to 
find with probability $1-\varepsilon$
 a polynomial $g$ such that $g\sim_e f$ 
 in time $e^{d}\(d \log q  \log (1/\varepsilon)\)^{O(1)}$ and 
 $O\(d\log q \log (1/\varepsilon)\)$ calls to $\Oef$.
\end{theorem}

The proofs of Theorems~\ref{thm:Quant} and~\ref{thm:Rand} are 
given below in Sections~\ref{sec:Quant} and~\ref{sec:Rand},
respectively.

\subsection{Coincidences among $e$th powers of 
polynomials}

The following result is immediate from the Weil bound on 
multiplicative character sums,  see~\cite[Theorem~11.23]{IwKow}.

\begin{lemma}
\label{lem:Coinc}  Let $g_1,g_2 \in \F_q[X]$ be two monic 
polynomials of degree at most $d$ with  $g_1\not \sim_e g_2$.
Then 
$$
\# \{x \in \F_q~:~ g_1(x)^e  =  g_2(x)^e\} 
= \frac{q}{e} + O(dq^{1/2}).
$$
\end{lemma}

We now immediately conclude. 

\begin{cor}
\label{cor:Coinc-e}  Let $g_1,g_2 \in \F_q[X]$ be two monic 
polynomials of degree  $o(q^{1/2})$ with  $g_1\not \sim_e g_2$.
Then for any $e \le (q-1)/2$ and a sufficiently large $q$
$$
\# \{x \in \F_q~:~ g_1(x)^e  \ne  g_2(x)^e\} \ge \frac{1}{3} q.
$$
\end{cor}

\subsection{Proof of Theorem~\ref{thm:Quant}}
\label{sec:Quant}

Let $\cS$ stand for the monic polynomials of degree at most $d$.
By Corollary~\ref{cor:Coinc-e}, a random choice of elements $x \in \F_q$ gives
with probability at least $0.99$
a set $T$ of size $O(\log |\cS|)=O(d\log q)$ 
such that for every pair $f,g\in \cS$ we have
$f(a)^e=g(a)^e$ for every $a\in T$ if and only if
$f\sim_e g$.

We continue with picking $d$ 
different elements $a_1,\ldots,a_d$ and use the oracle
$\Oef$ to obtain the values $b_j=f(a_j)^e$, $j=1,\ldots,d$, 
as well as to get the values $b(a)=f(a)^e$ for every $a\in T$. 

Using Shor's order finding and discrete logarithm algorithms~\cite{Shor}
we can also compute a generator $\zeta_e$ for the multiplicative
subgroup $\{u\in \F_q:u^e=1\}$ and for every $j$
an element $z_j\in \F_q$ such that $z_j^e=b_j$.

The cost of the steps performed so far is
polynomial in 
$\log q$ and $d$.
Let $E=\{0,\ldots,e-1\}$. For a tuple $\balpha=(\alpha_1,\ldots,\alpha_d)$
from $E^d$, let $f_{\balpha}$ be the monic 
polynomial of degree at most $d$ such that 
$f_{\balpha}(a_j)=z_j \zeta_e^{\alpha_j}$, $j=1,\ldots,d$.
For any specific tuple $\alpha$,
the polynomial $f_{\balpha}$ can be computed by simple interpolation
in time polynomial in $d\log q$.

We use Grover's search~\cite{Grover} over $E^d$ to find a tuple
$\balpha$ with probablity at least $0.99$ such that 
$f_{\balpha}^e(a)=b(a)$ for every $a\in T$.
The cost of this part is bounded by
$O(e^{d/2})$ times  
a polynomial in $\log q$ and $d$. 
Repeating the whole procedure $O(\log (1/\varepsilon))$
times
we achieve the desired probability level, which concludes the proof.

\subsection{Proof of Theorem~\ref{thm:Rand}}
\label{sec:Rand}
Observe that a generator for the group
$\{u\in \F_q:u^e=1\}$ as well as
elements $z_j$ with $z_j^e=b_j$ can be found
by simple classical algorithms of complexity
bounded by $e^{1/2}(\log q)^{O(1)}$, that is, 
even within the complexity bound of Theorem~\ref{thm:Quant}. 
Indeed, assume that for every prime $r$ diving $e$
we have an element $g_r\in \F_q$ which is not an $r$th 
power of an $\F_q$ element. 
Such elements can be found in time $(\log q)^{O(1)}$
using random choices.
The product of
appropriate powers of the elements $g_r$ is 
a generator for the group of the $e$th roots of unity. 

For computing an $e$th roots of $b_j$ it is sufficient to be able to take
$r$th root of an arbitrary field element $y$ for every prime divisor 
$r$ of $e$. This task can be accomplished in time $\sqrt{r}(\log q)^{O(1)}$ as in the algorithm of Adleman, Manders and Miller~\cite{AMM} instead of
the brute force one that uses Shanks' baby step-giant step method for
computing discrete logarithms in groups of order $r$, see~\cite[Section~5.3]{CrPom}.

Therefore, if we replace Grover's
search~\cite{Grover}  over $E^d$ with a classical search we obtain
a classical randomised algorithm of
complexity 
$e^d (d \log q \log(1/\varepsilon))^{O(1)}$.

\subsection{Further Remarks}
Under Generalised Riemann Hypothesis we can derandomize the proof of Theorem \ref{thm:Rand}. 	
If $q=p$ is a prime then a generator for the group of $e$th roots of
unity can be found in deterministic polynomial time.
If, furthermore, $e\le p^{\delta}$
or $e\le p^{\eta-\delta}$ 
for some fixed $\delta >0$, then
we could use the test of 
Theorem~\ref{thm:Small_e} or 
Theorem~\ref{thm:Medium_e}
to obtain a deterministic algorithm of complexity
$e^{d+c_0(d)\delta^{1/(2d-1)}}(d\log p)^{O(1)}$ or
$e^{d+\kappa+o(1)}(d\log p)^{O(1)}$, respectively.

\section{Comments and open problems}

One can obtain analogues of Theorems~\ref{thm:Small_e} and~\ref{thm:Medium_e}
in the settings of high degree extensions of finite fields.
More precisely, if $q= p^n$ for a fixed $p$ and growing $n$, we
write $\F_{q} \cong \F_p[X]/\left\langle\psi(X)\right\rangle$ for a fixed irreducible polynomial $\psi\in \F_p[X]$ of degree $n$. Then one can attempt to transfer the technique used in the proofs of Theorems~\ref{thm:Small_e} 
and~\ref{thm:Medium_e} to this case where a role of a short interval 
of length $h$ is now played by the set of polynomials of 
degree at most $h$. This approach has been used in~\cite{CillShp,Shp}
for several related problems. We also note that a version of 
effective Hilbert's Nullstellensatz for function fields, which is needed
for this approach, has recently been given by D'Andrea, Krick and Sombra~\cite{DKS}. 

We remark that we do not know how to take any advantage of actually 
knowing $g$, and get stronger version of Theorems~\ref{thm:Small_e} 
and~\ref{thm:Medium_e} in this case, 
like, for example, in~\cite[Section~3.2]{BGKS}.

\section*{Acknowledgement}

This research  was  supported in part by 
the Hungarian Scientific Research Fund (OTKA) Grant NK105645
(for G.I.); Singapore Ministry of Education and the 
National Research Foundation Tier 3 Grant MOE2012-T3-1-009
(for G.I. and M.S.); the Hausdorff Grant~EXC-59 (for M.K.); 
European Commission
IST STREP Project QALGO~600700 
and the French ANR Blanc Program Contract ANR-12-BS02-005 (for M.S.);  
Research-I Foundation CSE
and Hausdorff Center Bonn (for N.S.);
 the Australian Research Council Grant DP140100118 (for I.S.).

\bibliographystyle{amsalpha}
\bibliography{refs}
 
\end{document}